\documentclass[11pt,letterpaper,reqno]{amsart}
\usepackage{amsmath,amsthm,amsfonts,amssymb,mathtools}
\usepackage[lite]{amsrefs}
\usepackage[all]{xy}

\newtheorem{theorem}{Theorem}
\newtheorem{proposition}[theorem]{Proposition}

\numberwithin{equation}{section}

\begin{document}

\title{A T(1) theorem for entangled multilinear dyadic Calder\'{o}n-Zygmund operators}

\author{Vjekoslav Kova\v{c}}
\address{Vjekoslav Kova\v{c}, Department of Mathematics, University of Zagreb, Bijeni\v{c}ka cesta 30, 10000 Zagreb, Croatia}
\email{vjekovac@math.hr}
\thanks{The first author was partially supported by the MZOS grant 037-0372790-2799 of the Republic of Croatia}

\author{Christoph Thiele}
\address{Christoph Thiele, University of California at Los Angeles and
Universit\"at Bonn, Endenicher Allee 60, 53115 Bonn, Germany}
\email{thiele@math.uni-bonn.de}
\thanks{The second author was partially supported by NSF grant DMS 1001535}


\begin{abstract}
We prove a boundedness criterion for a class of dyadic multilinear forms acting on two-dimensional functions.
Their structure is more general than the one of classical multilinear Calder\'{o}n-Zygmund operators
as several functions can now depend on the same one-dimensional variable.
The study of this class is motivated by examples related to the two-dimensional bilinear Hilbert transform and to bilinear ergodic averages.
This paper is a sequel to a prior paper by the first author.
\end{abstract}

\maketitle

\tableofcontents

\section{Introduction}

The recent papers \cite{DT},\,\cite{Ber},\,\cite{Kov2},\,\cite{Kov1} initiated the study of multilinear singular integral operators and forms
with certain general modulation symmetries which force the Schwartz kernels
of the operators and forms to be supported on lower dimensional subspaces.
A multilinear variant of the Littlewood-Paley theory combined with the method of Bellman functions was introduced in \cite{Kov1} and \cite{Kov2} and applied
to proving $\mathrm{L}^p$ estimates for such multilinear forms acting on
functions of two variables.
The mentioned modulation symmetries appear when we can write the
multilinear form schematically as
\begin{equation}\label{eqentangledscheme}
\Lambda(F_1,F_2,\ldots) = \int_{\mathbb{R}^n} F_1({x_1},x_2) F_2({ x_1},x_3) \ldots \,K(x_1,\ldots,x_n)\,dx_1 dx_2 dx_3 \ldots dx_n ,
\end{equation}
with the exposed feature that $F_1$ and $F_2$ share the variable $x_1$.
If for some $g\in\mathrm{L}^{\infty}(\mathbb{R})$ we have
$$ \widetilde{F}_1(x,y) := g(x) F_1(x,y), \quad \widetilde{F}_2(x,y) := g(x) F_2(x,y), $$
then we have the modulation symmetry
$$ \Lambda(\widetilde{F}_1,F_2,\ldots) = \Lambda(F_1,\widetilde{F}_2,\ldots) . $$
Note that if two functions share both of their variables, which is tantamount
to a symmetry under arbitrary modulations by two dimensional functions,
then the two functions sharing both variables may be replaced by their pointwise product in the multilinear form.
The multilinear form then trivially reduces to one
of a lower degree. It is precisely the partial one-dimensional
modulation symmetries that cause multilinear forms (\ref{eqentangledscheme})
to require techniques outside the scope of the classical
Calder\'{o}n-Zygmund theory.

Our main result, Theorem \ref{theoremmain}, establishes
a T(1)-type criterion for boundedness of a class of multilinear forms of the above type.
In this theorem we specialize to a bipartite structure elaborated in the next section.
This bipartite structure manifests itself in that the variables fall into
two classes, namely the first entry variables $x_i$ and the second entry variables $y_j$.
Sharing of variables can only occur for variables of the same kind, so we rewrite (\ref{eqentangledscheme}) as
\begin{equation}\label{eqentangledscheme2}
\Lambda(F_1,F_2,\ldots) = \int_{\mathbb{R}^{m+n}} F_1(x_1,y_1) F_2(x_1,y_2) \ldots \,K(x_1,\ldots,x_m,y_1,\dots, y_n)\,dx_1 \ldots dx_m dy_1 \dots dy_n.
\end{equation}
If $m$ equals $n$, then the kernel $K$ should be thought of as the usual Coifman-Meyer
or multilinear Calder\'on-Zygmund kernel in $(\mathbb{R}^2)^n$, which is singular along the diagonal $(x_1,y_1)=\dots=(x_n,y_n)$.
The requirement $m=n$ is not necessary in our setting and the kernel is then singular along the set
\begin{equation}\label{diagonalequations}
x_1=\ldots=x_m,\ \ y_1=\ldots =y_n.
\end{equation}
We only discuss the dyadic model of the so-called perfect Calder\'{o}n-Zygmund kernels,
which are described in the next section.
At present we only have a partial understanding of the required
modifications for extending our results to the case of continuous Calder\'{o}n-Zygmund kernels.

The case $m=1$ or $n=1$ trivializes (see Subsection \ref{subseccounterex}), hence the smallest non-trivial and motivational
example of a form of the type (\ref{eqentangledscheme2}) writes as follows:
\begin{equation}\label{eqtwistedparaproduct}
\Lambda_{\sqcup}(F_1,F_2,F_3) := \int_{\mathbb{R}^4} F_1(x_1,y_1) F_2(x_1,y_2) F_3(x_2,y_1)
\,K(x_1,x_2,y_1,y_2) \,dx_1 dx_2 dy_1 dy_2  .
\end{equation}
The translation invariant case
$$K(x_1,x_2,y_1,y_2)=\kappa(x_1-x_2,y_1-y_2) $$
has been studied previously. It is a special case of the two-dimensional bilinear Hilbert transform introduced in \cite{DT}.
Estimates in various ranges of $\mathrm{L}^p$ spaces were proven by Bernicot \cite{Ber} and by the first author \cite{Kov2}.
The present paper particularly generalizes boundedness results for (\ref{eqtwistedparaproduct}) beyond translation invariance,
albeit only in the dyadic model of perfect Calder\'{o}n-Zygmund kernels.

We view this paper as partial progress towards
understanding some very challenging questions about operators with rather general structure (\ref{eqentangledscheme}).
Here we address the bipartite case (\ref{eqentangledscheme2}) only and it is much simpler than the general non-bipartite case.
A profound example of the latter is the \emph{triangular Hilbert transform}
\begin{equation}\label{eqtriangular}
\Lambda_{\triangle}(F_1,F_2,F_3) := \mathrm{p.v.}\int_{\mathbb{R}^3}
F_1(x,y)F_2(y,z)F_3(z,x)\, \frac{1}{x+y+z} \,dx dy dz .
\end{equation}
Compared with (\ref{eqtwistedparaproduct}) we have the additional
shared variable $x_2=y_2$ in the notation of (\ref{eqtwistedparaproduct}),
destroying the bipartite structure.
It appears possible that $\Lambda_{\triangle}$ satisfies some $\mathrm{L}^{p}$
bounds, for example a bound $\mathrm{L}^3\times \mathrm{L}^3\times \mathrm{L}^3\to \mathbb{C}$.
Proving such bounds remains out of reach at present. The form $\Lambda_{\triangle}$ is
closely related to and is stronger than both the one-dimensional bilinear Hilbert
transform and the Carleson operator, which can be realized by choosing the
functions $F_i$ in (\ref{eqtriangular}) appropriately.
The form $\Lambda_{\triangle}$ is also stronger than instances of (\ref{eqtwistedparaproduct}),
as the method of rotation deduces bounds on these instances from the conjectural bounds on (\ref{eqtriangular}).

Part of the motivation for (\ref{eqtriangular}) comes from the desire to
understand pointwise almost everywhere convergence and other questions
for ergodic averages of the form
$$ \frac{1}{n}\sum_{k=0}^{n-1}f(S^k\omega)g(T^k\omega), \ \ \omega\in\Omega, $$
for two commuting measure preserving transformations $S,T\colon\Omega\to\Omega$ on a probability space $\Omega$.
A real-analytic approach suggested by Demeter and the second author in \cite{DT} is to first understand estimates for the bilinear operator
$$ T_{\triangle}(F,G)(x,y) := \mathrm{p.v.}\int_{\mathbb{R}} F(x+t,y) G(x,y+t) \frac{dt}{t} . $$
If we dualize with another function $H$, we obtain
$$ \int_{\mathbb{R}^2}\Big(\mathrm{p.v.}\int_{\mathbb{R}} F(x+t,y) G(x,y+t) \frac{dt}{t}\Big) \,H(x,y) \,dx dy . $$
After the change of variables $z=-x-y-t$ and substitutions
$$ F_1(x,y)=H(x,y), \ \ F_2(y,z)=F(-y-z,y), \ \ F_3(z,x)=G(x,-x-z) , $$
we are left with the negative of the trilinear form (\ref{eqtriangular}).

The connection between $\Lambda_\triangle$ and the paradigmatic objects in
time frequency analysis, namely the bilinear Hilbert transform and Carleson's
operator, suggest to look at time frequency analysis as an approach to $\Lambda_{\triangle}$.
Time frequency analysis generally reduces estimation of more difficult
operators to estimation of simpler operators called tree models.
In the case of $\Lambda_\triangle$ we expect these simpler operators to be
forms of the type (\ref{eqtwistedparaproduct}), albeit with kernels somewhat
more general than those singular along the diagonal (\ref{diagonalequations}).

This paper, as a sequel to \cite{Kov1}, will use the same graph-theoretic setup as there and occasionally refer to results of \cite{Kov1}.
Theorems \ref{theoremmain} and \ref{theoremreformulated} below generalize \cite[Theorem 1.1]{Kov1}, apart from
changes in the exponent range.
We restrict attention to the two-dimensional case, that is to functions $F_i$
of two real variables, since we do not have phenomena to report in higher
dimensions other than straightforward generalizations. We refer to
\cite[Section 8]{Kov2} for a brief outline of some aspects of the higher dimensional theory.

\section{Formulation of the results}
\label{sectionformulation}

We write $A\lesssim B$ for two nonnegative quantities $A$ and $B$ if $A\leq C B$ holds with some constant $0\leq C<\infty$.
We will always specify if $C$ is an absolute constant or if it depends on some parameters.
In order to aid reading the text, functions of one real variable are denoted by lowercase letters,
while functions of several real variables are denoted by uppercase letters.
The characteristic function of a set $S\subseteq\mathbb{R}^n$ will be denoted by $\mathbf{1}_{S}$.

A \emph{dyadic interval} will always mean an interval of the form $[2^{-k}\ell,2^{-k}(\ell+1))$ for $k,\ell\in\mathbb{Z}$.
An \emph{$n$-dimensional dyadic cube} is any set $I_1\times I_2\times \ldots\times I_n$, where
$I_1,I_2,\ldots,I_n$ are dyadic intervals of the same length.
We will write $\mathcal{C}_n$ for the collection of all $n$-dimensional dyadic cubes.
Particularly important are \emph{dyadic squares} in $\mathbb{R}^{2}$, obtained by specializing $n=2$,
and their collection will be denoted simply by $\mathcal{C}$.

Fix two positive integers $m$ and $n$. The \emph{diagonal} in $\mathbb{R}^{m+n}$ is
$$ D := \big\{ (\underbrace{x,\ldots,x}_{m},\underbrace{y,\ldots,y}_{n}) : x,y\in\mathbb{R} \big\} . $$
The definition of \emph{perfect dyadic Calder\'{o}n-Zygmund kernel} will be adapted from \cite[Section 6]{AHMTT}.
It is a locally integrable function $K\colon\mathbb{R}^{m+n}\to\mathbb{C}$ satisfying the bound
\begin{equation}\label{eqczkernelest}
\big|K(x_1,\ldots,x_m,y_1,\ldots,y_n)\big| \,\lesssim\,
\Big(\sum_{1\leq i_1<i_2\leq m}\!\!|x_{i_1}\!-x_{i_2}| + \sum_{1\leq j_1<j_2\leq n}\!\!|y_{j_1}\!-y_{j_2}|\Big)^{2-m-n}
\end{equation}
for each $(x_1,\ldots,x_m,y_1,\ldots,y_n)\in \mathbb{R}^{m+n}\setminus D$ and such that $K$ is constant on each
$(m+n)$-dimensional dyadic cube \,$I_1\times\ldots\times I_m\times J_1\times\ldots\times J_n$\,
that does not intersect the diagonal.
Finally, we impose the qualitative technical condition that $K$ is bounded and compactly supported,
but without any quantitative information on either its $\mathrm{L}^\infty$-norm, or the size of its support.

Even though (\ref{eqczkernelest}) is a variant of the usual size estimate from multilinear Calder\'{o}n-Zygmund theory,
see the paper by Grafakos and Torres \cite{GT},
we emphasize that the operators we study in this paper can have more complicated structure due to the modulation symmetries.

The rest of the setup is similar to the one in \cite{Kov1}. Choose
$$ E \subseteq \{1,\ldots,m\} \times \{1,\ldots,n\} . $$
We can view $E$ as the set of edges of some simple bipartite undirected graph $G$ on the sets of vertices
$\{x_1,\ldots,x_m\}$ and $\{y_1,\ldots,y_n\}$.
More precisely, $x_i$ and $y_j$ are connected by an edge in this graph if and only if $(i,j)\in E$.
Finally, we define a multilinear form $\Lambda_E$ by
\begin{equation}\label{eqentangledform}
\Lambda_{E}\big((F_{i,j})_{(i,j)\in E}\big)
:= \int_{\mathbb{R}^{m+n}} K(x_1,\ldots,x_m,y_1,\ldots,y_n) \prod_{(i,j)\in E} \!F_{i,j}(x_i,y_j) \ dx_1 \ldots dx_m dy_1 \ldots dy_n .
\end{equation}
Thus, the graph structure determines which functions occur in the above expression.
To avoid degeneracy we assume that there are no isolated vertices in $G$, i.e.\@ each vertex is incident to some edge.
This forces each integration variable to appear in at least one of the
functions.
In all of the following we assume that the functions $F_{i,j}$ are bounded and compactly supported,
so that (\ref{eqentangledform}) is a priori well-defined.

The $|E|$-linear form $\Lambda_{E}$ uniquely determines $|E|$ mutually adjoint $(|E|\!-\!1)$-linear operators $T_{u,v}$, $(u,v)\in E$.
These are explicitly defined by
\begin{align}
& T_{u,v}\big((F_{i,j})_{(i,j)\in E\setminus\{(u,v)\}}\big)(x_u,y_v) \nonumber \\
& := \int_{\mathbb{R}^{m+n-2}} K(x_1,\ldots,x_m,y_1,\ldots,y_n) \prod_{(i,j)\in E\setminus\{(u,v)\}}
\!\!F_{i,j}(x_i,y_j) \ \prod_{i\neq u}dx_i \,\prod_{j\neq v}dy_j , \label{eqtuvoperators}
\end{align}
so that
\begin{equation}\label{eqduality}
\Lambda_{E}\big((F_{i,j})_{(i,j)\in E}\big)
= \int_{\mathbb{R}^2} T_{u,v}\big((F_{i,j})_{(i,j)\neq (u,v)}\big) \,F_{u,v} .
\end{equation}
Let us also recall the definition of the \emph{dyadic BMO-seminorm},
\begin{equation}\label{eqbmodef}
\|F\|_{\mathrm{BMO}(\mathbb{R}^2)}
:= \sup_{Q\in\mathcal{C}}\bigg(\frac{1}{|Q|} \int_{Q} \Big| F - \frac{1}{|Q|}\int_{Q}F \Big|^2 \,\bigg)^{1/2} .
\end{equation}

We are now ready to state the main result of this paper.

\begin{theorem}[Entangled T(1) theorem]\label{theoremmain}
\rule{1mm}{0mm}
\begin{itemize}
\item[(a)]
Suppose that $m,n\geq 2$.
There exist positive integers $(d_{i,j})_{(i,j)\in E}$, depending on $m,n,E$, satisfying
\begin{equation}\label{eqdscondition}
\sum_{(i,j)\in E}\frac{1}{d_{i,j}} > 1 ,
\end{equation}
such that the following holds.
If one has
\begin{equation}\label{eqt1wbpcondition}
|\Lambda_{E}(\mathbf{1}_{Q},\ldots,\mathbf{1}_{Q})|\lesssim |Q| \quad\textrm{for each } Q\in\mathcal{C}
\end{equation}
and
\begin{equation}\label{eqt1condition}
\big\|T_{u,v}(\mathbf{1}_{\mathbb{R}^2},\ldots,\mathbf{1}_{\mathbb{R}^2})\big\|_{\mathrm{BMO}(\mathbb{R}^2)}\lesssim 1
\quad\textrm{for each } (u,v)\in E ,
\end{equation}
then the form $\Lambda_{E}$ satisfies the estimate
\begin{equation}\label{eqt1estimate}
\big|\Lambda_{E}\big((F_{i,j})_{(i,j)\in E}\big)\big| \lesssim \prod_{(i,j)\in E} \|F_{i,j}\|_{\mathrm{L}^{p_{i,j}}(\mathbb{R}^2)}
\end{equation}
whenever exponents $p_{i,j}$ are taken from the region determined by
$$ \sum_{(i,j)\in E} \frac{1}{p_{i,j}}=1 \quad\textrm{and}\quad d_{i,j}<p_{i,j}\leq\infty\textrm{ for each } (i,j)\in E. $$
In particular, it is guaranteed by \emph{(\ref{eqdscondition})}
that Estimate \emph{(\ref{eqt1estimate})} holds in a non-empty open range of exponents.
The implicit constant in \emph{(\ref{eqt1estimate})} depends on implicit constants from
\emph{(\ref{eqczkernelest})},\emph{(\ref{eqt1wbpcondition})},\emph{(\ref{eqt1condition})},
on the graph structure (which was given by $m,n,E$), and on the exponents $p_{i,j}$.

\item[(b)]
Conversely, if Inequality \emph{(\ref{eqt1estimate})} holds for some choice of exponents $1<p_{i,j}\leq\infty$
such that $\sum_{(i,j)\in E}p_{i,j}^{-1}=1$,
then Conditions \emph{(\ref{eqt1wbpcondition})} and \emph{(\ref{eqt1condition})} must also be satisfied,
with constants depending on the constants from \emph{(\ref{eqczkernelest})},\emph{(\ref{eqt1estimate})},
on the graph, and on the exponents.
\end{itemize}
\end{theorem}

For a first insight into concrete choices of $d_{i,j}$ the reader is referred
to (\ref{eqchoiceofds1}), which applies whenever the $d_{i,j}$ such defined
satisfy (\ref{eqdscondition}). This happens in most instances, the few
exceptional cases are discussed in Section \ref{sectionexponents}.
Generally, the exponent region increases as $|E|$ gets larger with fixed $m,n$.
The actual range of estimates is not known even for very simple instances of entangled forms, such as $\Lambda_{\sqcup}$,
and it is clear that known techniques are insufficient to address the question of optimal range.
Disheartened by this fact we did not attempt to make the region of exponents
as large as possible within our set of available techniques.

Condition (\ref{eqt1wbpcondition}) is a version of a
\emph{weak boundedness property}. Usually weak boundedness properties involve
testing with bump functions (see \cite{DJ}),
but here it is natural to test characteristic functions of dyadic squares.
The T(1)-type conditions in (\ref{eqt1condition}) are similar to requirements of the classical T(1) theorem from \cite{DJ},
but we need to perform $|E|$ verifications, instead of only two.
One might compare this situation with the T(1) theorem in \cite{BDNTTV}, dealing with modulation-invariant trilinear forms
and having three conditions of that kind.

The emphasis of Theorem \ref{theoremmain} is on the qualitative $\mathrm{L}^p$ bound for $\Lambda_E$ and on constant dependencies,
since mere continuity of $\Lambda_E$ is automatic by our assumptions on $K$.
In applications one would apply the theorem to truncated and localized kernels $K$.

We can restate Theorem \ref{theoremmain} in terms of \emph{local T(1) conditions}, compare with \cite[Corollary 6.3]{AHMTT}.
Their continuous analogues are the well-known \emph{restricted boundedness conditions} from \cite{Ste}.
In the following formulation we will also emphasize the generalized modulation invariance we mentioned in the Introduction.
Suppose that we are given two one-dimensional functions $a_{Q}^{i,j},b_{Q}^{i,j}\in\mathrm{L}^{\infty}(\mathbb{R})$
for each dyadic square $Q=I\times J\in\mathcal{C}$ and for each $(i,j)\in E$.
We require that
\begin{equation}\label{eqasandbs}
\prod_{j \,:\, (i,j)\in E}\!\!a_{Q}^{i,j} = \mathbf{1}_{I} \,\textrm{ for every } i,
\qquad \prod_{i \,:\, (i,j)\in E}\!\!b_{Q}^{i,j} = \mathbf{1}_{J} \,\textrm{ for every } j ,
\end{equation}
each $a_{Q}^{i,j}$ vanishes outside $I$, and each $b_{Q}^{i,j}$ vanishes outside $J$.
To simplify the notation we also write
$$ B_{Q}^{i,j} := a_{Q}^{i,j} \otimes b_{Q}^{i,j} , $$
where $f\otimes g$ denotes the elementary tensor defined by $(f\otimes g)(x,y):=f(x)g(y)$.

\begin{theorem}[Entangled T(1) theorem, reformulation]\label{theoremreformulated}
For $m,n\geq 2$ there exist positive integers $(d_{i,j})_{(i,j)\in E}$ depending on $m,n,E$,
satisfying \emph{(\ref{eqdscondition})}, and such that the following holds.
If one can verify
\begin{equation}\label{eqt1restrictedorig}
\big\|T_{u,v}\big((B_{Q}^{i,j})_{(i,j)\in E\setminus\{(u,v)\}}\big)
B_{Q}^{u,v} \big\|_{\mathrm{L}^{1}(\mathbb{R}^2)}\lesssim |Q|
\end{equation}
for each $Q=I\times J$ and for each $(u,v)\in E$,
then Estimate \emph{(\ref{eqt1estimate})} holds in the same range of exponents as in Theorem \ref{theoremmain}.
The implicit constant depends on implicit constants from \emph{(\ref{eqczkernelest})},\emph{(\ref{eqt1restrictedorig})},
on $m,n,E$, and on the exponents $p_{i,j}$.
\end{theorem}

Note that (\ref{eqasandbs}) can be equivalently written as
$$ \prod_{(i,j)\in E} a_{Q}^{i,j}(x_i) b_{Q}^{i,j}(y_j) = \prod_{i=1}^{m}\mathbf{1}_{I}(x_i) \prod_{j=1}^{n}\mathbf{1}_{J}(y_j)  . $$
By the one-dimensional modulation invariance of (\ref{eqtuvoperators}) we instantly observe
$$ T_{u,v}\big(a_{Q}^{i,j}\!\otimes\!b_{Q}^{i,j}\big)_{(i,j)\in E\setminus\{(u,v)\}}
\,\big(a_{Q}^{u,v}\!\otimes\!b_{Q}^{u,v}\big)
= T_{u,v}(\mathbf{1}_{I}\!\otimes\!\mathbf{1}_{J})_{(i,j)\in E\setminus\{(u,v)\}}(\mathbf{1}_{I}\!\otimes\!\mathbf{1}_{J}) , $$
so that Condition (\ref{eqt1restrictedorig}) equivalently reads
\begin{equation}\label{eqt1restricted}
\big\|T_{u,v}(\mathbf{1}_{Q},\ldots,\mathbf{1}_{Q})\big\|_{\mathrm{L}^{1}(Q)}\lesssim |Q| .
\end{equation}
In the following sections we can assume that this reduction is already performed.
An interesting consequence of Theorem \ref{theoremreformulated}, as it is formulated, is that in verification of the Bound (\ref{eqt1estimate}),
it is enough to assume that all functions $F_{i,j}$ but one are indeed elementary tensors,
in which case the entangled structure disappears.
On the other hand, Theorem \ref{theoremreformulated} is not really a T(b)-type theorem, as the functions $B_{Q}^{i,j}$ have to be chosen
under quite special constraints.

Let us turn to particular instances of our general result, assuming that Conditions (\ref{eqt1wbpcondition}) and (\ref{eqt1condition}) are satisfied.
If we take $m=n=2$ and $E=\{(1,1),(1,2),(2,1)\}$, we obtain (\ref{eqtwistedparaproduct}), written as
$$ \Lambda_{\sqcup}(F_{1,1},F_{1,2},F_{2,1})
= \int_{\mathbb{R}^4} F_{1,1}(x_1,y_1) F_{1,2}(x_1,y_2) F_{2,1}(x_2,y_1) \,K(x_1,x_2,y_1,y_2) \, dx_1 dx_2 dy_1 dy_2 . $$
The corresponding graph is depicted in the left half of Figure \ref{figureexamples}.
By taking a look at the numbers $d_{i,j}$ defined by (\ref{eqchoiceofds1}), we see that Estimate (\ref{eqt1estimate}) holds in the range
\,$p_{1,1}^{-1}+p_{1,2}^{-1}+p_{2,1}^{-1}=1$, \,$p_{1,1},p_{1,2},p_{2,1}>2$.
On the other hand, there is no good dyadic model for (\ref{eqtriangular}) that would be covered by Theorem \ref{theoremmain}.
The closest in spirit is the ``four-cyclic'' quadrilinear form
\begin{align*}
& \Lambda_{\Box}(F_{1,1},F_{1,2},F_{2,1},F_{2,2}) \\
& = \int_{\mathbb{R}^4} F_{1,1}(x_1,y_1) F_{1,2}(x_1,y_2) F_{2,1}(x_2,y_1) F_{2,2}(x_2,y_2) \,K(x_1,x_2,y_1,y_2) \, dx_1 dx_2 dy_1 dy_2 .
\end{align*}
It is determined by $m=n=2$ and $E=\{1,2\}\!\times\!\{1,2\}$; see the right half of Figure \ref{figureexamples}.
This time we establish the Bound (\ref{eqt1estimate}) for
\,$p_{1,1}^{-1}+p_{1,2}^{-1}+p_{2,1}^{-1}+p_{2,2}^{-1}=1$, \,$p_{1,1},p_{1,2},p_{2,1},p_{2,2}>2$.
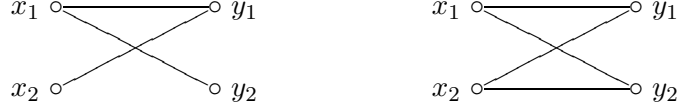
\begin{figure}
$$ \xymatrix @C=1.2mm @M=0mm @W=0mm { x_1 & \circ \ar@{-}[rrrrrrrrrrrrrrrr] \ar@{-}[drrrrrrrrrrrrrrrr] &&&&&&&&&&&&&&&& \circ & y_1 \\
x_2 & \circ \ar@{-}[urrrrrrrrrrrrrrrr] &&&&&&&&&&&&&&&& \circ & y_2 }
\qquad\qquad\qquad
\xymatrix @C=1.2mm @M=0mm @W=0mm { x_1 & \circ \ar@{-}[rrrrrrrrrrrrrrrr] \ar@{-}[drrrrrrrrrrrrrrrr] &&&&&&&&&&&&&&&& \circ & y_1 \\
x_2 & \circ \ar@{-}[urrrrrrrrrrrrrrrr] \ar@{-}[rrrrrrrrrrrrrrrr] &&&&&&&&&&&&&&&& \circ & y_2 } $$
\caption{Bipartite graphs associated with $\Lambda_{\sqcup}$ and $\Lambda_{\Box}$ respectively.}
\label{figureexamples}
\end{figure}

We also comment that if $m=n$ and each component of $G$ contains only one edge, then (after relabeling) we can write
$$ \Lambda_{E}(F_{1,1},\ldots,F_{n,n}) = \int_{(\mathbb{R}^2)^n} K(x_1,\ldots,x_n,y_1,\ldots,y_n)
\prod_{j=1}^{n} F_{j,j}(x_j,y_j) \,dx_1 dy_1\ldots dx_n dy_n . $$
In this case Theorems \ref{theoremmain} and \ref{theoremreformulated}
recover dyadic variants of classical results from multilinear Calder\'{o}n-Zygmund theory, as it was developed by Grafakos and Torres \cite{GT}.
As is seen from (\ref{eqchoiceofds1}), we reprove Estimate
(\ref{eqt1estimate}) in this case whenever
\,$\sum_{j=1}^{n}p_{j,j}^{-1}=1$\, and $p_{j,j}>1$ for each $j$.

As for the organization of the paper,
Sections \ref{sectionsufficiency} and \ref{sectionexponents} establish Part (a) of Theorem \ref{theoremmain} in all possible cases,
while Section \ref{sectionnecessity} proves Theorem \ref{theoremreformulated} and then reduces Part (b) of Theorem \ref{theoremmain} to it.
In Section \ref{sectionexponents} we also explain why we require $m,n\geq 2$ by giving an appropriate counterexample in the case $\min\{m,n\}=1$.

\section{Sufficiency of the testing conditions}
\label{sectionsufficiency}

The most substantial argument is the proof of Part (a) of Theorem \ref{theoremmain} and we present it in this section.
By multilinearity of $\Lambda_E$ we can assume that functions $F_{i,j}$ are nonnegative.

We begin by introducing quantities that will determine the range of exponents in which $\mathrm{L}^p$ estimates for $\Lambda_E$ will be proven.
They are obtained by splitting $G$ into connected components $G_1,G_2,\ldots,G_k$.
Vertices of each component $G_\ell$ of $G$ are conveniently indexed by the sets
$$ \mathcal{X}_{\ell} := \big\{i\in\{1,\ldots,m\} : x_i\textrm{ is a vertex in }G_\ell\big\}
\ \ \textrm{and}\ \ \mathcal{Y}_{\ell} := \big\{j\in\{1,\ldots,n\} : y_j\textrm{ is a vertex in }G_\ell\big\} , $$
while its edges are determined by $E_{\ell} := E\cap(\mathcal{X}_{\ell}\times\mathcal{Y}_{\ell})$.
Take a connected component $G_\ell$ and for each $(i,j)\in E_\ell$ set
\begin{equation}\label{eqchoiceofds1}
d_{i,j} := \left\{\begin{array}{cl}
\max\{|\mathcal{X}_\ell|,|\mathcal{Y}_\ell|\} & \textrm{ if }E_\ell=\mathcal{X}_{\ell}\times\mathcal{Y}_{\ell}
\textrm{ or }\max\{|\mathcal{X}_\ell|,|\mathcal{Y}_\ell|\}\leq 2, \\
\max\{|\mathcal{X}_\ell|,|\mathcal{Y}_\ell|\}+1 & \textrm{ otherwise}.
\end{array}\right.
\end{equation}
Observe that $E_\ell=\mathcal{X}_{\ell}\times\mathcal{Y}_{\ell}$ simply means that the component $G_\ell$ is a complete bipartite graph.
We keep this choice throughout this section and it will be sufficient for ``most'' graphs $G$,
i.e.\@ for all but a single series of exceptional cases.
In Section \ref{sectionexponents} we will discuss non-triviality of that exponent range and treat those exceptional instances differently.

\subsection{Decomposition into paraproducts}
We write $\mathbf{h}_{I}$ for the $\mathrm{L}^\infty$-normalized Haar wavelet on a dyadic interval $I$,
i.e.\@ $\mathbf{h}_{I}=\mathbf{1}_{I_\mathrm{left}}-\mathbf{1}_{I_\mathrm{right}}$,
where $I_\mathrm{left}$ and $I_\mathrm{right}$ denote the two halves of $I$.
Consider an orthonormal basis of $\mathrm{L}^{2}(\mathbb{R}^{m+n})$ consisting of the functions
$$ \frac{\mathbf{a}^{(1)}_{I_1}}{|I_1|^{1/2}}\otimes\ldots\otimes\frac{\mathbf{a}^{(m)}_{I_m}}{|I_m|^{1/2}}
\otimes\frac{\mathbf{b}^{(1)}_{J_1}}{|J_1|^{1/2}}\otimes\ldots\otimes\frac{\mathbf{b}^{(n)}_{J_n}}{|J_n|^{1/2}} , $$
where $I_1\times\ldots\times I_m\times J_1\times\ldots\times J_n$ ranges over all $(m+n)$-dimensional dyadic cubes and
each of the letters $\mathbf{a}^{(i)}$ and $\mathbf{b}^{(j)}$ stands for either the characteristic function $\mathbf{1}$
or the Haar function $\mathbf{h}$, excluding the possibility that all of them are simultaneously $\mathbf{1}$.
Since $K$ is a square integrable function by our qualitative assumptions, it can be decomposed in the above basis as
\begin{align*}
& K(x_1,\ldots,x_m,y_1,\ldots,y_n) \\
& = \sum_{\substack{I_1\times\ldots\times I_m\times J_1\times\ldots\times J_n\in\mathcal{C}_{m+n} \\
\mathbf{a}^{(1)},\ldots,\mathbf{a}^{(m)},\,\mathbf{b}^{(1)},\ldots,\mathbf{b}^{(n)}\in\{\mathbf{1},\mathbf{h}\} \\
\mathbf{a}^{(i)},\,\mathbf{b}^{(j)}\textrm{ are not all equal } \mathbf{1} }}
\nu_{I_1\times\ldots\times I_m\times J_1\times\ldots\times J_n}^{\mathbf{a}^{(1)},\ldots,\mathbf{a}^{(m)},\mathbf{b}^{(1)},\ldots,\mathbf{b}^{(n)}}
\ \mathbf{a}^{(1)}_{I_1}(x_1)\ldots \mathbf{a}^{(m)}_{I_m}(x_m) \mathbf{b}^{(1)}_{J_1}(y_1)\ldots \mathbf{b}^{(n)}_{J_n}(y_n) ,
\end{align*}
where
\begin{equation}\label{eqkerneldecomp}
\nu_{I_1\times\ldots\times I_m\times J_1\times\ldots\times J_n}^{\mathbf{a}^{(1)},\ldots,\mathbf{a}^{(m)},\mathbf{b}^{(1)},\ldots,\mathbf{b}^{(n)}}
:= \bigg\langle K\,,\, \frac{\mathbf{a}^{(1)}_{I_1}}{|I_1|}\otimes\ldots\otimes\frac{\mathbf{a}^{(m)}_{I_m}}{|I_m|}
\otimes\frac{\mathbf{b}^{(1)}_{J_1}}{|J_1|}\otimes\ldots\otimes\frac{\mathbf{b}^{(n)}_{J_n}}{|J_n|} \bigg\rangle_{\mathrm{L}^2(\mathbb{R}^{m+n})}
\end{equation}
and $\langle\cdot,\cdot\rangle_{\mathrm{L}^2}$ denotes the standard inner product.
Now we use the perfect cancellation condition. Since $K$ is constant on cubes that do not intersect the diagonal,
the corresponding coefficients (\ref{eqkerneldecomp}) are zero, so only the terms with $I_1=\ldots=I_m$ and $J_1=\ldots=J_n$ remain.
This leads to a representation of $\Lambda_{E}$ as a finite sum of \emph{entangled dyadic paraproducts},
\begin{equation}\label{eqdecompparaprod}
\Lambda_{E} = \sum_{\substack{\mathbf{a}^{(1)},\ldots,\mathbf{a}^{(m)},\,\mathbf{b}^{(1)},\ldots,\mathbf{b}^{(n)}\in\{\mathbf{1},\mathbf{h}\} \\
\mathbf{a}^{(i)},\,\mathbf{b}^{(j)}\textrm{ are not all equal } \mathbf{1} }}
\Theta_{E}^{\mathbf{a}^{(1)},\ldots,\mathbf{a}^{(m)},\mathbf{b}^{(1)},\ldots,\mathbf{b}^{(n)}} .
\end{equation}
These are defined explicitly by
\begin{align}
\Theta\big((F_{i,j})_{(i,j)\in E}\big) = & \
\Theta_{E}^{\mathbf{a}^{(1)},\ldots,\mathbf{a}^{(m)},\mathbf{b}^{(1)},\ldots,\mathbf{b}^{(n)}}\big((F_{i,j})_{(i,j)\in E}\big) \nonumber \\
:= & \sum_{I\times J\in\mathcal{C}} \lambda_{I\times J} \ |I|^{2-m-n} \int_{\mathbb{R}^{m+n}} \prod_{(i,j)\in E} F_{i,j}(x_i,y_j) \nonumber \\
& \qquad \mathbf{a}^{(1)}_{I}(x_1)\ldots \mathbf{a}^{(m)}_{I}(x_m) \mathbf{b}^{(1)}_{J}(y_1)\ldots \mathbf{b}^{(n)}_{J}(y_n)
\,dx_1 \ldots dx_m dy_1 \ldots dy_n , \label{eqparaproductsdef}
\end{align}
with coefficients
\begin{align}
\lambda_{I\times J} = & \ \lambda_{I\times J}^{\mathbf{a}^{(1)},\ldots,\mathbf{a}^{(m)},\mathbf{b}^{(1)},\ldots,\mathbf{b}^{(n)}} \nonumber \\
:= & \ |I|^{m+n-2}
\,\nu_{I\times\ldots\times I\times J\times\ldots\times J}^{\mathbf{a}^{(1)},\ldots,\mathbf{a}^{(m)},\mathbf{b}^{(1)},\ldots,\mathbf{b}^{(n)}} \nonumber \\
= & \ |I\times J|^{-1} \big\langle K\,,\, \mathbf{a}^{(1)}_{I}\!\otimes\!\ldots\!\otimes \mathbf{a}^{(m)}_{I}
\!\otimes \mathbf{b}^{(1)}_{J}\!\otimes\!\ldots\!\otimes \mathbf{b}^{(n)}_{J} \big\rangle_{\mathrm{L}^2(\mathbb{R}^{m+n})} . \label{eqcoefficientsdef}
\end{align}
Since we will be working with a fixed choice of $\mathbf{a}^{(1)},\ldots,\mathbf{a}^{(m)},\mathbf{b}^{(1)},\ldots,\mathbf{b}^{(n)}$ most of the time,
we omit denoting that dependence in the following text and simply write $\Theta$ and $\lambda_{I\times J}$.
From Decomposition (\ref{eqdecompparaprod}) we see that it is enough to prove Bound (\ref{eqt1estimate})
for each entangled dyadic paraproduct $\Theta$.

Let us introduce some convenient notation.
The average value of a function $f$ on an interval $I$ will be denoted by $[f(x)]_{x\in I}$,
while the same average of the function $f\mathbf{h}_I$ will be denoted by $\langle f(x)\rangle_{x\in I}$.
In other words,
$$ [f(x)]_{x\in I} := \frac{1}{|I|}\int_{\mathbb{R}}f(x)\mathbf{1}_{I}(x)dx,
\qquad \langle f(x)\rangle_{x\in I} := \frac{1}{|I|}\int_{\mathbb{R}}f(x)\mathbf{h}_{I}(x)dx . $$
We will be dealing with multi-variable expressions, so it will be important to clarify in which variable the average is taken.
To save space we allow joining several averaging procedures into a single bracket notation, with all averaging
variables listed in its subscript.
If we set
$$ {\setlength{\arraycolsep}{2pt}\begin{array}{llll}
S & := \big\{ i\in\{1,\ldots,m\} : \mathbf{a}^{(i)}=\mathbf{h} \big\}, & \quad T & := \big\{ j\in\{1,\ldots,n\} : \mathbf{b}^{(j)}=\mathbf{h} \big\}, \\
S^c & := \big\{ i\in\{1,\ldots,m\} : \mathbf{a}^{(i)}=\mathbf{1} \big\}, & \quad T^c & := \big\{ j\in\{1,\ldots,n\} : \mathbf{b}^{(j)}=\mathbf{1} \big\},
\end{array}} $$
then the paraproduct given in (\ref{eqparaproductsdef}) can be written as
$$ \Theta\big((F_{i,j})_{(i,j)\in E}\big) =
\sum_{I\times J\in\mathcal{C}} \lambda_{I\times J} \, |I\times J| \, \mathcal{A}_{I\times J}\big((F_{i,j})_{(i,j)\in E}\big) , $$
where
\begin{equation}\label{eqparaproductterm}
\mathcal{A}_{I\times J}\big((F_{i,j})_{(i,j)\in E}\big) := \bigg[ \bigg\langle \prod_{(i,j)\in E} F_{i,j}(x_i,y_j)
\bigg\rangle_{\substack{x_i\in I\mathrm{\,for\,each\,}i\in S\\ y_j\in J\mathrm{\,for\,each\,}j\in T}}
\,\bigg]_{\substack{x_i\in I\mathrm{\,for\,each\,}i\in S^c\\ y_j\in J\mathrm{\,for\,each\,}j\in T^c}} .
\end{equation}

We continue with some quite standard terminology.
\emph{Child} and \emph{parent} relations on dyadic squares are defined as usually, so that each square has four children and exactly one parent.
A \emph{finite convex tree} is any finite collection $\mathcal{T}$ of dyadic squares that is convex with respect to the set inclusion
and that contains the largest square (i.e.\@ a square that covers all others).
It is denoted $Q_\mathcal{T}$ and called the \emph{tree-top} of $\mathcal{T}$.
\emph{Leaves} of $\mathcal{T}$ are then defined to be dyadic squares that are not members of $\mathcal{T}$
but their parents are. The family of leaves is denoted $\mathcal{L}(\mathcal{T})$.
One can define a localized absolute version of each $\Theta$ for a finite convex tree $\mathcal{T}$ simply as
$$ \Theta_{\mathcal{T}} := \sum_{Q\in\mathcal{T}} |\lambda_{Q}| \,|Q| \,|\mathcal{A}_{Q}| . $$

By a classical stopping time argument one can reduce the desired $\mathrm{L}^p$ estimate for $\Theta$ to a ``single tree estimate'',
\begin{equation}\label{eqsingletreeest}
\Theta_{\mathcal{T}}\big((F_{i,j})_{(i,j)\in E}\big) \,\lesssim\, |Q_\mathcal{T}|
\prod_{(i,j)\in E} \max_{Q\in\mathcal{T}\cup\mathcal{L}(\mathcal{T})}\big[F_{i,j}^{d_{i,j}}\big]_{Q}^{1/d_{i,j}}
\end{equation}
for each finite convex tree $\mathcal{T}$ with tree-top $Q_\mathcal{T}$.
Here $[F]_Q$ denotes the average value of a function $F$ on a square $Q$.
This reduction can be found in \cite[Section 4]{Kov1}, which is tailored exactly to this situation,
as all arguments given there are fairly general.
We only remark that \cite{Kov1} does not explicitly mention any estimates involving $\mathrm{L}^\infty(\mathbb{R}^2)$, but these are also easily derived.
If $p_{i,j}=\infty$ for some $(i,j)\in E$, then for those pairs $(i,j)$ we simply control $\big[F_{i,j}^{d_{i,j}}\big]_{Q}^{1/d_{i,j}}$ in
(\ref{eqsingletreeest}) by $\|F_{i,j}\|_{\mathrm{L}^\infty(\mathbb{R}^2)}$ and apply the stopping time procedure
to the remaining functions.

In order to establish (\ref{eqsingletreeest}) we need to distinguish two essentially different cases of entangled dyadic paraproducts.
We call $\Theta$ \emph{non-cancellative} if there exists an edge in $G$ that is incident to all vertices from
$\{x_i : i\in S\} \cup \{y_j : j\in T\}$. Otherwise we say that $\Theta$ is \emph{cancellative}.
In order to clarify this concept let us observe that cancellative paraproducts $\Theta$ correspond to the cases
\begin{itemize}
\item[(C1)] $\max\{|S|,|T|\}\geq 2$,
\item[(C2)] $S=\{i\}$, $T=\{j\}$, and $(i,j)\not\in E$,
\end{itemize}
while non-cancellative paraproducts appear in the cases
\begin{itemize}
\item[(NC1)] $|S|+|T|=1$,
\item[(NC2)] $S=\{i\}$, $T=\{j\}$, and $(i,j)\in E$.
\end{itemize}
Recall that \,$S=T=\emptyset$\, is not possible because then all $\mathbf{a}^{(i)}$ and $\mathbf{b}^{(j)}$ would be $\mathbf{1}$.

\subsection{Cancellative paraproducts}
Types (C1) and (C2) will be resolved by the following proposition.

\begin{proposition}\label{propositioncancellative}
Suppose that $\Theta$ is a cancellative entangled dyadic paraproduct defined by \emph{(\ref{eqparaproductsdef})}
with coefficients \emph{(\ref{eqcoefficientsdef})}.
\begin{itemize}
\item[(a)]
If Condition \emph{(\ref{eqt1wbpcondition})} holds, then the coefficients $\lambda=(\lambda_{Q})_{Q\in\mathcal{C}}$ satisfy
$$ \|\lambda\|_{\ell^{\infty}} := \sup_{Q\in\mathcal{C}}|\lambda_Q| \,\lesssim 1 . $$
\item[(b)]
Each localized form $\Theta_{\mathcal{T}}$ for a finite convex tree $\mathcal{T}$ satisfies
$$ \Theta_{\mathcal{T}}\big((F_{i,j})_{(i,j)\in E}\big) \,\lesssim\, \|\lambda\|_{\ell^{\infty}} |Q_\mathcal{T}|\!
\prod_{(i,j)\in E} \max_{Q\in\mathcal{T}\cup\mathcal{L}(\mathcal{T})}\!\big[F_{i,j}^{d_{i,j}}\big]_{Q}^{1/d_{i,j}} . $$
\end{itemize}
\end{proposition}

\begin{proof}[Proof of Proposition \ref{propositioncancellative}]\rule{1mm}{0mm}

\ (a) \
Take an arbitrary dyadic square $I\times J\in\mathcal{C}$.
Let us begin by substituting $F_{i,j}=\mathbf{1}_{I}\otimes\mathbf{1}_{J}$ for each $(i,j)\in E$ into (\ref{eqentangledform}).
Since there are no isolated vertices in $G$, the characteristic function $\mathbf{1}_{I}(x_i)$ is attached to each variable $x_i$
(possibly several times) and the same holds for the variables $y_j$.
Thus, Equation (\ref{eqentangledform}) becomes
$$ \Lambda_{E}\big(\underbrace{\mathbf{1}_{I}\!\otimes\!\mathbf{1}_{J},\ldots,\mathbf{1}_{I}\!\otimes\!\mathbf{1}_{J}}_{|E|}\big)
= \big\langle K \,,\, \underbrace{\mathbf{1}_{I}\!\otimes\!\ldots\!\otimes\!\mathbf{1}_{I}}_{m}\otimes
\underbrace{\mathbf{1}_{J}\!\otimes\!\ldots\!\otimes\!\mathbf{1}_{J}}_{n} \big\rangle_{\mathrm{L}^{2}(\mathbb{R}^{m+n})} , $$
so Inequality (\ref{eqt1wbpcondition}) gives
\begin{equation}\label{eqauxkernelest}
\big|\big\langle K \,,\, \mathbf{1}_{I}\!\otimes\!\ldots\!\otimes\!\mathbf{1}_{I}\!\otimes\!
\mathbf{1}_{J}\!\otimes\!\ldots\!\otimes\!\mathbf{1}_{J} \big\rangle_{\mathrm{L}^{2}(\mathbb{R}^{m+n})}\big| \lesssim |I\times J| .
\end{equation}

We can split each of \,$\mathbf{a}^{(1)}_{I},\ldots,\mathbf{a}^{(m)}_{I}$ as $\mathbf{1}_{I_\mathrm{left}}\pm\mathbf{1}_{I_\mathrm{right}}$
and each of \,$\mathbf{b}^{(1)}_{J},\ldots,\mathbf{b}^{(n)}_{J}$ as $\mathbf{1}_{J_\mathrm{left}}\pm\mathbf{1}_{J_\mathrm{right}}$.
That way we can estimate $|\lambda_{I\times J}|$ by
\begin{equation}\label{eqauxcoeffest}
|I\times J|^{-1}\!\!\!\sum_{\substack{I^{(1)},\ldots,I^{(m)}\in\{I_\mathrm{left},I_\mathrm{right}\} \\
J^{(1)},\ldots,J^{(n)}\in\{J_\mathrm{left},J_\mathrm{right}\}}}\!\!
\big|\big\langle K \,,\, \mathbf{1}_{I^{(1)}}\!\otimes\!\ldots\!\otimes\!\mathbf{1}_{I^{(m)}}\!\otimes\!
\mathbf{1}_{J^{(1)}}\!\otimes\!\ldots\!\otimes\!\mathbf{1}_{J^{(n)}} \big\rangle_{\mathrm{L}^{2}(\mathbb{R}^{m+n})}\big| .
\end{equation}
The terms with $I^{(1)}\!=\!\ldots\!=\!I^{(m)}$, $J^{(1)}\!=\!\ldots\!=\!J^{(n)}$
can be controlled using (\ref{eqauxkernelest}) applied to each of the four subsquares
$$ I_\mathrm{left}\times J_\mathrm{left}, \ \, I_\mathrm{left}\times J_\mathrm{right}, \ \,
I_\mathrm{right}\times J_\mathrm{left}, \ \, I_\mathrm{right}\times J_\mathrm{right} $$
to yield
$$ |I\times J|^{-1}\!\!\sum_{\substack{I'\in\{I_\mathrm{left},I_\mathrm{right}\}\\ J'\in\{J_\mathrm{left},J_\mathrm{right}\}}}\!\!
\big|\big\langle K \,,\, \mathbf{1}_{I'}\!\otimes\!\ldots\!\otimes\!\mathbf{1}_{I'}\!\otimes\!
\mathbf{1}_{J'}\!\otimes\!\ldots\!\otimes\!\mathbf{1}_{J'} \big\rangle_{\mathrm{L}^{2}(\mathbb{R}^{m+n})}\big| \,\lesssim 1 . $$
Each of the $2^{m+n}-4$ remaining terms in (\ref{eqauxcoeffest}) will be estimated individually.
Without loss of generality suppose \,$m\geq 2$,\, $I^{(1)}=I_\mathrm{left}$,\, $I^{(2)}=I_\mathrm{right}$\,
and denote the left endpoint of $I_\mathrm{right}$ by $x_0$.
If $x_1\in I_\mathrm{left}$, $x_2\in I_\mathrm{right}$, $x_3,\ldots,x_m\in I$, $y_1,\ldots,y_n\in J$, then
\begin{align*}
|x_1-x_2| & = |x_1-x_0| + |x_2-x_0| \\
|x_i-x_1| + |x_i-x_2| & \geq |x_i-x_0|, \quad i=3,\ldots,m
\end{align*}
and consequently also
\begin{align*}
& \sum_{1\leq i_1<i_2\leq m}\!\!|x_{i_1}\!-x_{i_2}| + \sum_{1\leq j_1<j_2\leq n}\!\!|y_{j_1}\!-y_{j_2}| \\
& \geq \,\sum_{i=1}^{m}|x_i-x_0| + \sum_{j=1}^{n-1}|y_j-y_n| \,\geq \Big(\sum_{i=1}^{m}(x_i-x_0)^2 + \sum_{j=1}^{n-1}(y_j-y_n)^2\Big)^{1/2} .
\end{align*}
Passing to the spherical coordinates in $\mathbb{R}^{m+n-1}$ and using the kernel size estimate (\ref{eqczkernelest}) we obtain
\begin{align*}
& |I\times J|^{-1}\big|\big\langle K \,,\, \mathbf{1}_{I_\mathrm{left}}\!\otimes\!\mathbf{1}_{I_\mathrm{right}}\!\otimes\!\mathbf{1}_{I^{(3)}}
\!\otimes\ldots\!\otimes\!\mathbf{1}_{I^{(m)}}\!\otimes\!
\mathbf{1}_{J^{(1)}}\!\otimes\!\ldots\!\otimes\!\mathbf{1}_{J^{(n)}} \big\rangle_{\mathrm{L}^{2}(\mathbb{R}^{m+n})}\big| \\
& \leq |I|^{-2} \int_{J} \Big(\int_{\mathrm{Ball}((x_0,y_n),(m+n)|I|)}
|K(x_1,\ldots,x_m,y_1,\ldots,y_{n-1},y_n)| \,dx_1\ldots dx_m dy_1\ldots dy_{n-1}\Big) dy_n \\
& \leq |I|^{-2} \int_{J} \Big( \int_{0}^{(m+n)|I|} r^{2-m-n}\, r^{m+n-2} dr \Big) dy_n \ \lesssim \ 1 .
\end{align*}
Thus, $|\lambda_{I\times J}|\lesssim 1$.
The implicit constant certainly depends on $m,n$ and the constants from (\ref{eqczkernelest}),\,(\ref{eqt1wbpcondition}),
but these dependencies are understood.

\ (b) \
We begin by estimating $|\lambda_Q|\leq\|\lambda\|_{\ell^{\infty}}$, so we can indeed normalize $\|\lambda\|_{\ell^{\infty}}=1$
and then abandon the coefficients by estimating them from above by $1$.
Since the expression $\mathcal{A}_Q$ is scale-invariant,
one can even re-scale the tree $\mathcal{T}$ to achieve $|Q_\mathcal{T}|=1$ and then (using homogeneity) also normalize the functions by
\begin{equation}\label{eqnormalization}
\max_{Q\in\mathcal{T}\cup\mathcal{L}(\mathcal{T})}\!\big[F_{i,j}^{d_{i,j}}\big]_{Q}^{1/d_{i,j}} = 1
\end{equation}
for each $(i,j)\in E$.
Let us realize the structural splitting with respect to the graph components.
We denote
$$ \mathcal{A}_{I\times J}^{(\ell)} := \bigg[ \bigg\langle \prod_{(i,j)\in E_{\ell}} \!\! F_{i,j}(x_i,y_j)
\bigg\rangle_{\substack{x_i\in I\mathrm{\,for\,each\,}i\in S\cap\mathcal{X}_{\ell}\\ y_j\in J\mathrm{\,for\,each\,}j\in T\cap\mathcal{Y}_{\ell}}}
\,\bigg]_{\substack{x_i\in I\mathrm{\,for\,each\,}i\in S^c\cap\mathcal{X}_{\ell}\\ y_j\in J\mathrm{\,for\,each\,}j\in T^c\cap\mathcal{Y}_{\ell}}} , $$
so that $\mathcal{A}_{Q} = \prod_{\ell=1}^{k} \mathcal{A}_{Q}^{(\ell)}$ for each dyadic square $Q=I\times J$.

The paper \cite{Kov1} was devoted exactly to entangled dyadic paraproducts that fall into Class (C1),
although this terminology was not used there.
In Case (C1) Estimate (\ref{eqsingletreeest}) is covered precisely by Proposition 3.1 in \cite{Kov1}.
However, it is important to notice that there is even an improvement over our claim: it establishes (\ref{eqsingletreeest})
when $d_{i,j}$ is simply replaced with $\max\{|\mathcal{X}_\ell|,|\mathcal{Y}_\ell|\}$ for every $(i,j)\in E_\ell$,
i.e.\@ there is no need to add $1$ to it.
It is worth observing that quantities $[F_{i,j}^d]_{Q}^{1/d}$ are increasing in $d$ by Jensen's inequality.

In what follows we show how to handle entangled paraproducts from Class (C2), i.e.\@ those that have $S=\{i_0\}$, $T=\{j_0\}$, $(i_0,j_0)\not\in E$.
Such $i_0$ and $j_0$ will be fixed for the whole succeeding discussion. We further split into two subcases.

\emph{Case 1.} The vertices $x_{i_0}$ and $y_{j_0}$ of $G$ belong to the same component.

Without loss of generality we can assume $i_0\in\mathcal{X}_1$, $j_0\in\mathcal{Y}_1$.
After Normalization (\ref{eqnormalization}) and multiple applications of H\"{o}lder's inequality or Lemma 3.2 from \cite{Kov1},
we can bound $|\mathcal{A}_{Q}^{(\ell)}|\leq 1$ for each $Q\in\mathcal{T}$ and $\ell\neq 1$.
This in turn gives $|\mathcal{A}_{Q}|\leq |\mathcal{A}_{Q}^{(1)}|$, so we can work with the component $G_1$ only.
Compare this reduction with \cite[Subsection 3.2]{Kov1}.

Exploiting the fact $(i_0,j_0)\not\in E$ we can rewrite $\mathcal{A}_{Q}^{(1)}$ as
$$ \mathcal{A}_{Q}^{(1)} = \bigg[ \Big\langle \prod_{j \,:\, (i_0,j)\in E_1} \!\!\! F_{i_0,j}(x_{i_0},y_j) \Big\rangle_{x_{i_0}\in I}
\Big\langle \prod_{i \,:\, (i,j_0)\in E_1} \!\!\! F_{i,j_0}(x_i,y_{j_0}) \Big\rangle_{y_{j_0}\in J}
\prod_{\substack{(i,j)\in E_1 \\ i\neq i_0,\, j\neq j_0}} \!\! F_{i,j}(x_i,y_j)
\bigg]_{\substack{x_i\in I\mathrm{\,for\,}i\in\mathcal{X}_1\setminus\{i_0\}\\ y_j\in J\mathrm{\,for\,}j\in\mathcal{Y}_1\setminus\{j_0\}}} $$
and then estimate $|\mathcal{A}_{Q}^{(1)}|\leq\frac{1}{2}\widetilde{\mathcal{A}}_{Q}+\frac{1}{2}\bar{\mathcal{A}}_{Q}$, where
\begin{align*}
\widetilde{\mathcal{A}}_{Q} & := \bigg[ \Big\langle \prod_{j \,:\, (i_0,j)\in E_1} \!\!\! F_{i_0,j}(x_{i_0},y_j) \Big\rangle_{x_{i_0}\in I}^2
\prod_{\substack{(i,j)\in E_1 \\ i\neq i_0,\, j\neq j_0}} \!\! F_{i,j}(x_i,y_j)
\bigg]_{\substack{x_i\in I\mathrm{\,for\,}i\in\mathcal{X}_1\setminus\{i_0\}\\ y_j\in J\mathrm{\,for\,}j\in\mathcal{Y}_1\setminus\{j_0\}}} , \\
\bar{\mathcal{A}}_{Q} & := \bigg[ \Big\langle \prod_{i \,:\, (i,j_0)\in E_1} \!\!\! F_{i,j_0}(x_i,y_{j_0}) \Big\rangle_{y_{j_0}\in J}^2
\prod_{\substack{(i,j)\in E_1 \\ i\neq i_0,\, j\neq j_0}} \!\! F_{i,j}(x_i,y_j)
\bigg]_{\substack{x_i\in I\mathrm{\,for\,}i\in\mathcal{X}_1\setminus\{i_0\}\\ y_j\in J\mathrm{\,for\,}j\in\mathcal{Y}_1\setminus\{j_0\}}} .
\end{align*}
Consider the new entangled paraproducts $\widetilde{\Theta}$ and $\bar{\Theta}$ with local versions defined by
$$ \widetilde{\Theta}_{\mathcal{T}} := \sum_{Q\in\mathcal{T}} |Q| \,\widetilde{\mathcal{A}}_{Q}
\quad\textrm{and}\quad \bar{\Theta}_{\mathcal{T}} := \sum_{Q\in\mathcal{T}} |Q| \,\bar{\mathcal{A}}_{Q} . $$
Observe that $\widetilde{\Theta}$ corresponds to the graph obtained from $G_1$ by ``doubling'' the vertex $x_{i_0}$ and deleting the vertex $y_{j_0}$.
This way we obtain a (connected or disconnected) graph with at most $\max\{|\mathcal{X}_1|\!+\!1,|\mathcal{Y}_1|\!-\!1\}$ vertices in each bipartition.
It belongs to the Class (C1), so we can invoke \cite[Proposition 3.1]{Kov1} again.
Analoguously we proceed with $\bar{\Theta}_{\mathcal{T}}$, which leads us to
$$ \Theta_\mathcal{T}\big((F_{i,j})_{(i,j)\in E}\big)
\leq \frac{1}{2}\widetilde{\Theta}_{\mathcal{T}}\big((F_{i,j})_{(i,j)\in E}\big)
+ \frac{1}{2}\bar{\Theta}_{\mathcal{T}}\big((F_{i,j})_{(i,j)\in E}\big) \lesssim 1 . $$

Note that in \cite[Proposition 3.1]{Kov1} the thresholds $d_{i,j}$ are
defined to be $\max\{|\mathcal{X}_1|,|\mathcal{Y}_1|\}$, but since we
apply this proposition to a modified connected component with component
size changed by $1$, we only obtain the desired estimate with the
threshold $\max\{|\mathcal{X}_1|,|\mathcal{Y}_1|\}+1$ with respect to the
original graph.
However, this is all we claimed when $\max\{|\mathcal{X}_1|,|\mathcal{Y}_1|\}\geq 3$ and the component in question is not a complete graph, and thus
proves the desired bound for $\Theta_\mathcal{T}$ in this case.

Under the assumption of Case 1 the graph $G_1$ is never complete.
It remains to consider the case when $G_1$  satisfies
$\max\{|\mathcal{X}_1|,|\mathcal{Y}_1|\}\leq 2$. After relabeling if necessary,
this implies under the assumption of Case 1 that
$E_1 = \{(1,1),(1,2),(2,1)\}$, $S=\{2\}$, $T=\{2\}$. This graph was
already depicted in the left half of Figure \ref{figureexamples}.
Using the Cauchy-Schwarz inequality, $Q=I\times J\in\mathcal{T}$, and $[F_{1,1}^2]_{Q}^{1/2}\leq 1$,
{\allowdisplaybreaks\begin{align*}
\big|\mathcal{A}_{Q}^{(1)}\big| & = \big|\big[\langle F_{2,1}(x_2,y_1)\rangle_{x_2\in I}\,
\langle F_{1,2}(x_1,y_2)\rangle_{y_2\in J}\, F_{1,1}(x_1,y_1)\big]_{x_1\in I,\,y_1\in J}\big| \\
& \leq \big[\langle F_{2,1}(x_2,y_1)\rangle_{x_2\in I}^2\,\langle F_{1,2}(x_1,y_2)\rangle_{y_2\in J}^2\big]_{x_1\in I,\,y_1\in J}^{1/2}
\big[F_{1,1}(x_1,y_1)^2\big]_{x_1\in I,\,y_1\in J}^{1/2} \\
& = \big[\langle F_{2,1}(x_2,y_1)\rangle_{x_2\in I}^2\big]_{y_1\in J}^{1/2}
\big[\langle F_{1,2}(x_1,y_2)\rangle_{y_2\in J}^2\big]_{x_1\in I}^{1/2}
\big[F_{1,1}(x_1,y_1)^2\big]_{x_1\in I,\,y_1\in J}^{1/2} \\
& \leq \frac{1}{2}\big[\langle F_{2,1}(x_2,y_1)\rangle_{x_2\in I}^2\big]_{y_1\in J}
+\frac{1}{2}\big[\langle F_{1,2}(x_1,y_2)\rangle_{y_2\in J}^2\big]_{x_1\in I} .
\end{align*}}
This time we ended up with paraproducts
$$ \widetilde{\Theta}_{\mathcal{T}} := \sum_{I\times J\in\mathcal{T}} |I\times J| \big[\langle F_{2,1}(x,y)\rangle_{x\in I}^2\big]_{y\in J} ,
\quad \bar{\Theta}_{\mathcal{T}} := \sum_{I\times J\in\mathcal{T}} |I\times J| \big[\langle F_{1,2}(x,y)\rangle_{y\in J}^2\big]_{x\in I} , $$
corresponding to two graphs on only three vertices.
They are both in the Class (C1), which completes the discussion of Case 1.

\emph{Case 2.} The vertices $x_{i_0}$ and $y_{j_0}$ of $G$ belong to different components.

Without loss of generality we can assume that $i_0\in\mathcal{X}_1$ and $j_0\in\mathcal{Y}_2$.
By \cite[Lemma 3.2]{Kov1} again we can estimate
$$ |\mathcal{A}_{Q}| = \prod_{\ell=1}^{k} |\mathcal{A}_{Q}^{(\ell)}| \leq |\mathcal{A}_{Q}^{(1)}| |\mathcal{A}_{Q}^{(2)}|
\leq \frac{1}{2}\big(\mathcal{A}_{Q}^{(1)}\big)^2 + \frac{1}{2}\big(\mathcal{A}_{Q}^{(2)}\big)^2 $$
for each $Q\in\mathcal{T}$.
This in turn immediately reduces Bound (\ref{eqsingletreeest}) for $\Theta_\mathcal{T}$ to the same bound for
$$ \widetilde{\Theta}_{\mathcal{T}} := \sum_{Q\in\mathcal{T}} |Q| \big(\mathcal{A}_{Q}^{(1)}\big)^2
\quad\textrm{and}\quad \bar{\Theta}_{\mathcal{T}} := \sum_{Q\in\mathcal{T}} |Q| \big(\mathcal{A}_{Q}^{(2)}\big)^2 . $$
Both of these are indeed localized versions of entangled paraproducts falling into the Class (C1).
More precisely, $\widetilde{\Theta}$ corresponds to a graph with two components, each being a copy of $G_1$,
so Proposition 3.1  from \cite{Kov1} can be applied once again and it yields (\ref{eqsingletreeest})
with all exponents equal to $\max\{|\mathcal{X}_1|,|\mathcal{Y}_1|\}$.
Similarly we deal with $\bar{\Theta}$. Note that we have proved a sharper variant of (\ref{eqsingletreeest}) again, i.e.\@ the one when
$d_{i,j}$ is replaced with $\max\{|\mathcal{X}_\ell|,|\mathcal{Y}_\ell|\}$ for $(i,j)\in E_\ell$.
\end{proof}

\subsection{Non-cancellative paraproducts}
Now we handle paraproduct types (NC1) and (NC2).

\begin{proposition}\label{propositionnoncancellative}
Suppose that $\Theta$ is a non-cancellative entangled dyadic paraproduct defined by \emph{(\ref{eqparaproductsdef})}
with coefficients \emph{(\ref{eqcoefficientsdef})}.
\begin{itemize}
\item[(a)]
If Conditions \emph{(\ref{eqt1condition})} hold, then the coefficients $\lambda=(\lambda_{Q})_{Q\in\mathcal{C}}$ satisfy
$$ \|\lambda\|_{\mathrm{bmo}} :=
\sup_{Q_0\in\mathcal{C}}\Big(\frac{1}{|Q_0|}\sum_{Q\in\mathcal{C}\,:\,Q\subseteq Q_0}\!|Q||\lambda_Q|^2\Big)^{1/2} \lesssim 1 . $$
\item[(b)]
Each localized form $\Theta_{\mathcal{T}}$ for a finite convex tree $\mathcal{T}$ satisfies
$$ \Theta_{\mathcal{T}}\big((F_{i,j})_{(i,j)\in E}\big) \,\lesssim\, \|\lambda\|_{\mathrm{bmo}} |Q_\mathcal{T}|\!
\prod_{(i,j)\in E} \max_{Q\in\mathcal{T}\cup\mathcal{L}(\mathcal{T})}\!\big[F_{i,j}^{d_{i,j}}\big]_{Q}^{1/d_{i,j}} . $$
\end{itemize}
\end{proposition}

\begin{proof}[Proof of Proposition \ref{propositionnoncancellative}]\rule{1mm}{0mm}

\ (a) \
Let us begin by fixing $(u,v)\in E$ and computing \,$T_{u,v}(\mathbf{1}_{\mathbb{R}^2},\ldots,\mathbf{1}_{\mathbb{R}^2})$.
We achieve this by substituting $F_{i,j}=\mathbf{1}_{\mathbb{R}^2}$ for $(i,j)\neq (u,v)$ in
$\Lambda_{E}((F_{i,j})_{(i,j)\in E})$ and using Representation (\ref{eqdecompparaprod}).
In this case
\begin{align*}
& \Theta_{E}^{\mathbf{a}^{(1)},\ldots,\mathbf{a}^{(m)},\mathbf{b}^{(1)},\ldots,\mathbf{b}^{(n)}}\big((F_{i,j})_{(i,j)\in E}\big) \\
& = \sum_{I\times J\in \mathcal{C}} \lambda_{I\times J}^{\mathbf{a}^{(1)},\ldots,\mathbf{a}^{(m)},\mathbf{b}^{(1)},\ldots,\mathbf{b}^{(n)}} |I|^{2-m-n}
\Big(\prod_{i\neq u}\int_{\mathbb{R}}\mathbf{a}_{I}^{(i)}\Big) \Big(\prod_{j\neq v}\int_{\mathbb{R}}\mathbf{b}_{J}^{(j)}\Big)
\int_{\mathbb{R}^2} F_{u,v}(x,y) \mathbf{a}_{I}^{(u)}(x) \mathbf{b}_{J}^{(v)}(y) dx dy ,
\end{align*}
which can only be non-zero when $S\subseteq\{u\}$ and $T\subseteq\{v\}$.
Let us rather denote the coefficients by $\lambda_{I\times J}^{S,T}$, so that
$$ \Lambda_E \big((F_{i,j})_{(i,j)\in E}\big)
= \!\sum_{I\times J\in \mathcal{C}} \int_{\mathbb{R}^2} \!F_{u,v}(x,y)
\Big( \lambda_{I\times J}^{\emptyset,\{v\}}\mathbf{1}_{I}(x)\mathbf{h}_{J}(y)
+ \lambda_{I\times J}^{\{u\},\emptyset}\mathbf{h}_{I}(x)\mathbf{1}_{J}(y)
+ \lambda_{I\times J}^{\{u\},\{v\}}\mathbf{h}_{I}(x)\mathbf{h}_{J}(y) \Big) dx dy . $$
Since $F_{u,v}$ can be chosen arbitrarily, using (\ref{eqduality}) we have obtained
$$ T_{u,v}(\mathbf{1}_{\mathbb{R}^2},\ldots,\mathbf{1}_{\mathbb{R}^2})
= \sum_{I\times J\in \mathcal{C}} \Big( \lambda_{I\times J}^{\emptyset,\{v\}}\,\mathbf{1}_{I}\otimes\mathbf{h}_{J}
+ \lambda_{I\times J}^{\{u\},\emptyset}\,\mathbf{h}_{I}\otimes\mathbf{1}_{J}
+ \lambda_{I\times J}^{\{u\},\{v\}}\,\mathbf{h}_{I}\otimes\mathbf{h}_{J} \Big) . $$
From (\ref{eqbmodef}) one easily concludes
\begin{equation}\label{eqbmoest}
\big\|T_{u,v}(\mathbf{1}_{\mathbb{R}^2},\ldots,\mathbf{1}_{\mathbb{R}^2})\big\|_{\mathrm{BMO}(\mathbb{R}^2)}
= \sup_{Q_0\in\mathcal{C}} \Big( \frac{1}{|Q_0|}\sum_{Q\in\mathcal{C}\,:\,Q\subseteq Q_0} \!|Q|\,
\big( \big|\lambda_{Q}^{\emptyset,\{v\}}\big|^2 \!+\! \big|\lambda_{Q}^{\{u\},\emptyset}\big|^2
\!+\! \big|\lambda_{Q}^{\{u\},\{v\}}\big|^2 \big) \Big)^{1/2} .
\end{equation}
Combining Condition (\ref{eqt1condition}) and Equality (\ref{eqbmoest}) we obtain
$$ \big\|\lambda^{\emptyset,\{v\}}\big\|_{\mathrm{bmo}},\, \big\|\lambda^{\{u\},\emptyset}\big\|_{\mathrm{bmo}},\,
\big\|\lambda^{\{u\},\{v\}}\big\|_{\mathrm{bmo}} \,\lesssim 1 . $$
It remains to observe that for each non-cancellative paraproduct $\Theta$ there exists $(u,v)\in E$
such that $S\subseteq\{u\}$ and $T\subseteq\{v\}$, simply by definition.
Therefore its coefficients satisfy the desired bound.

\ (b) \
Again we normalize to achieve $|Q_\mathcal{T}|=1$ and (\ref{eqnormalization}).
Proceed by applying the Cauchy-Schwarz inequality to get
$$ \Theta_\mathcal{T} \leq \Big(\sum_{Q\in\mathcal{T}} |Q| |\lambda_Q|^2\Big)^{1/2}
\Big(\sum_{Q\in\mathcal{T}} |Q| \big(\mathcal{A}_{Q}\big)^2\Big)^{1/2} . $$
On the one hand,
$$ \sum_{Q\in\mathcal{T}} |Q| |\lambda_Q|^2
\leq \!\sum_{Q\in\mathcal{C}\,:\,Q\subseteq Q_\mathcal{T}}\!\! |Q| |\lambda_Q|^2
\leq \|\lambda\|_\mathrm{bmo}^2 |Q_\mathcal{T}| = \|\lambda\|_\mathrm{bmo}^2 . $$
On the other hand,
$$ \widetilde{\Theta}_{\mathcal{T}} := \sum_{Q\in\mathcal{T}} |Q| \big(\mathcal{A}_{Q}\big)^2 $$
is an example of a local version of an entangled paraproduct falling into the Class (C1).
Its graph is obtained by simply doubling all connected components of $G$.
From Proposition 3.1 of \cite{Kov1} we get $\widetilde{\Theta}_{\mathcal{T}}\lesssim 1$,
so $\Theta_\mathcal{T}\lesssim\|\lambda\|_\mathrm{bmo}$, which completes the proof.
\end{proof}

Let us comment once again that Propositions \ref{propositioncancellative} and \ref{propositionnoncancellative} together
prove Estimate (\ref{eqsingletreeest}) for each entangled paraproduct $\Theta$, which in turn has to be combined
with a stopping time argument and Decomposition (\ref{eqdecompparaprod}) to establish Estimate (\ref{eqt1estimate}).

\section{Discussion of the exponent range and the exceptional cases}
\label{sectionexponents}

The purpose of this section is to ensure that the results of Theorems \ref{theoremmain} and \ref{theoremreformulated}
are not void for any non-trivial graph $G$, i.e.\@ whenever $m,n\geq 2$.
We have already established Part (a) of Theorem \ref{theoremmain} for ``most'' of the graphs and now we have to
identify and resolve the several exceptional cases.
Begin by observing that exponents $p_{i,j}$ such that $\sum_{(i,j)\in E}p_{i,j}^{-1}=1$ and $d_{i,j}<p_{i,j}<\infty$
will certainly exist as soon as (\ref{eqdscondition}) is satisfied.

\subsection{Non-emptiness of the range}
Each component $G_\ell$ has at least $|\mathcal{X}_\ell|+|\mathcal{Y}_\ell|-1$ edges, as otherwise it would not be a connected graph.
If $\max\{|\mathcal{X}_\ell|,|\mathcal{Y}_\ell|\}\leq 2$, then the component contributes to the sum in (\ref{eqdscondition}) with at least
$$ \frac{|\mathcal{X}_\ell|+|\mathcal{Y}_\ell|-1}{\max\{|\mathcal{X}_\ell|,|\mathcal{Y}_\ell|\}} \geq 1 , $$
while if $\max\{|\mathcal{X}_\ell|,|\mathcal{Y}_\ell|\}\geq 3$, then it adds at least
$$ \frac{|\mathcal{X}_\ell|+|\mathcal{Y}_\ell|-1}{\max\{|\mathcal{X}_\ell|,|\mathcal{Y}_\ell|\}+1} \geq \frac{3}{4} . $$
We conclude that (\ref{eqdscondition}) surely holds when $G$ has two or more components.

If $G$ is connected, then by $m,n\geq 2$ we have
$$ \frac{m+n-1}{\max\{m,n\}} > 1 \quad\textrm{and}\quad
\frac{m+n-1}{\max\{m,n\}+1} = 1 + \frac{\min\{m,n\}-2}{\max\{m,n\}+1} , $$
so (\ref{eqdscondition}) can fail only when $\min\{m,n\}=2$ and $\max\{m,n\}\geq 3$.
Moreover, if (\ref{eqdscondition}) is false, then
$$ 1 \leq \frac{|E|}{m+n-1} = \frac{|E|}{\max\{m,n\}+1} = \sum_{(i,j)\in E}\frac{1}{d_{i,j}}\leq 1 , $$
which gives $|E|=m+n-1$. We see that the number of edges in $G$ is by $1$ smaller than the number of vertices,
which implies that $G$ is a tree (i.e.\@ it has no cycles), by a well-known result from graph theory.

Therefore, the choice (\ref{eqchoiceofds1}) guarantees nontrivial $\mathrm{L}^p$ estimates in
Theorems \ref{theoremmain} and \ref{theoremreformulated} in
all but the following exceptional case:
$\min\{m,n\}=2$, $\max\{m,n\}\geq 3$, and $G$ is a tree.
By symmetry we suppose $m=2<n$ and then vertices $x_1$ and $x_2$ must have precisely one common neighbor among the $y_j$'s,
in order for the graph to stay connected and to avoid cycles.
By relabeling the vertices we can suppose that there is $r\in\{1,\ldots,n\}$, $r\geq\lceil\frac{n+1}{2}\rceil\geq 2$ such that
$$ E = \{(1,1),(1,2),\ldots,(1,r),(2,r),\ldots,(2,n)\} . $$

For $r\neq n$ we can modify the proof of Proposition \ref{propositioncancellative} to obtain the single tree estimate with
$d_{i,j}:=n$ for each $(i,j)\in E$.
The only cancellative paraproducts that are not covered by \cite[Proposition 3.1]{Kov1} are the ones from Class (C2).
Without loss of generality $S=\{1\}$ and $T=\{n\}$. Then we can control the paraproduct term as
\begin{align}
|\mathcal{A}_Q| & = \bigg| \bigg[ \Big\langle \prod_{j=1}^{r}F_{1,j}(x_1,y_j) \Big\rangle_{x_1\in I}
\big\langle F_{2,n}(x_2,y_n)\big\rangle_{y_n\in J}
\prod_{j=r}^{n-1} F_{2,j}(x_2,y_j) \bigg]_{x_2\in I,\ y_1,\ldots,y_{n-1}\in J} \bigg| \nonumber \\
& \leq \frac{1}{2}\bigg[ \Big\langle \prod_{j=1}^{r}F_{1,j}(x_1,y_j) \Big\rangle_{x_1\in I}^2
\prod_{j=r}^{n-1} F_{2,j}(x_2,y_j) \bigg]_{x_2\in I,\ y_1,\ldots,y_{n-1}\in J} \label{eqexception1} \\
& \quad + \frac{1}{2}\bigg[ \big\langle F_{2,n}(x_2,y_n)\big\rangle_{y_n\in J}^2 \prod_{j=r}^{n-1} F_{2,j}(x_2,y_j)
\bigg]_{x_2\in I,\ y_r,\ldots,y_{n-1}\in J} . \label{eqexception2}
\end{align}
Note that (\ref{eqexception1}) and (\ref{eqexception2}) lead to (possibly disconnected) graphs having at most
$$ \max\{3,\,n-1,\,n-r+2,\,r+1\} \leq n $$
vertices in each of their bipartition vertex-sets, so \cite[Proposition 3.1]{Kov1} can be applied once again.
On the other hand, the proof of Proposition \ref{propositionnoncancellative}, which handles non-cancellative paraproducts, can be left unchanged.
This in turn establishes (\ref{eqt1estimate}) whenever Conditions (\ref{eqt1wbpcondition}) and (\ref{eqt1condition}) are satisfied.

Finally, the only case we left out is when $r=n$.
An example of such graph is shown in Figure \ref{figureg3graph}.
This time we rather take
\begin{equation}\label{eqnewregion}
d_{1,j}:=n \,\textrm{ for } j=1,\ldots,n-1, \quad d_{1,n}:=2n-2, \quad d_{2,n}:=n .
\end{equation}
Cancellative paraproducts of Type (C1) are again resolved by \cite[Proposition 3.1]{Kov1}, even in a slightly larger range.
All paraproducts in Class (C2) are essentially the same, so choose $S=\{2\}$, $T=\{1\}$ and begin by estimating
\begin{align*}
|\mathcal{A}_Q| & = \bigg|\bigg[ \big\langle F_{1,1}(x_1,y_1)\big\rangle_{y_1\in J} \big\langle F_{2,n}(x_2,y_n)\big\rangle_{x_2\in I}
\prod_{j=2}^{n} F_{1,j}(x_1,y_j) \bigg]_{x_1\in I,\ y_2,\ldots,y_{n}\in J}\bigg| \\
& \leq \big[\langle F_{1,1}(x_1,y_1)\rangle_{y_1\in J}^2\big]_{x_1\in I}^{1/2}
\big[\langle F_{2,n}(x_2,y_n)\rangle_{x_2\in I}^2\big]_{y_n\in J}^{1/2}
\bigg[ \prod_{j=2}^{n} F_{1,j}(x_1,y_j)^2 \bigg]_{x_1\in I,\ y_2,\ldots,y_{n}\in J}^{1/2} .
\end{align*}
From (\ref{eqnormalization}) and
$$ \bigg[ \prod_{j=2}^{n} F_{1,j}(x_1,y_j)^2 \bigg]_{x_1\in I,\ y_2,\ldots,y_{n}\in J}^{1/2}
\leq \,\prod_{j=2}^{n} \big[ F_{1,j}(x_1,y_j)^{2n-2} \big]_{x_1\in I,\, y_j\in J}^{1/(2n-2)} $$
we obtain
$$ |\mathcal{A}_Q| \leq \frac{1}{2}\big[\langle F_{1,1}(x_1,y_1)\rangle_{y_1\in J}^2\big]_{x_1\in I}
+\frac{1}{2}\big[\langle F_{2,n}(x_2,y_n)\rangle_{x_2\in I}^2\big]_{y_n\in J}, $$
which this time yields the single tree estimate (\ref{eqsingletreeest}) with an unusual choice
$\tilde{d}_{1,j}=2n-2$ for each $1\leq j\leq n$ and $\tilde{d}_{2,n}=2$.
Consequently, Bound (\ref{eqt1estimate}) holds for $p_{i,j}>\tilde{d}_{i,j}$.
We are going to apply the fiberwise Calder\'{o}n-Zygmund decomposition of Bernicot \cite{Ber}
to expand the exponent region so that it contains the one determined by (\ref{eqnewregion}).
This is important because all entangled paraproducts corresponding to the same multilinear form
should satisfy the same $\mathrm{L}^p$ estimate in order for this estimate to hold for the form $\Lambda_E$ itself.
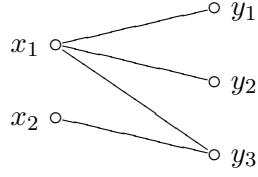
\begin{figure}
$$ \xymatrix @C=1.2mm @R=2.5mm @M=0mm @W=0mm {
& &&&&&&&&&&&&&&&& \circ & y_1 \\
x_1 & \circ \ar@{-}[urrrrrrrrrrrrrrrr] \ar@{-}[drrrrrrrrrrrrrrrr] \ar@{-}[dddrrrrrrrrrrrrrrrr] &&&&&&&&&&&&&&&& & \\
& &&&&&&&&&&&&&&&& \circ & y_2 \\
x_2 & \circ \ar@{-}[drrrrrrrrrrrrrrrr] &&&&&&&&&&&&&&&& & \\
& &&&&&&&&&&&&&&&& \circ & y_3
} $$
\caption{Example of an exceptional bipartite graph.}
\label{figureg3graph}
\end{figure}

We only give a short comment on Bernicot's argument, as we apply it in almost exactly the same way as in \cite[Section 5]{Kov2}.
By choosing $(u,v)=(1,n)$ and by duality, the problem becomes to extend estimates
$$ \|T_{1,n}(F_{2,n},F_{1,1},\ldots,F_{1,n-1})\|_{\mathrm{L}^{p}(\mathbb{R}^2)}
\lesssim \|F_{2,n}\|_{\mathrm{L}^{p_{2,n}}(\mathbb{R}^2)} \prod_{j=1}^{n-1}\|F_{1,j}\|_{\mathrm{L}^{p_{1,j}}(\mathbb{R}^2)} $$
for the operator
\begin{align*}
T_{1,n}(F_{2,n},F_{1,1},\ldots,F_{1,n-1})(x,y) & = \sum_{I\times J\in \mathcal{C}} \lambda_{I\times J} |I|^{-n}
\Big(\int_{\mathbb{R}}F_{2,n}(x',y)\mathbf{a}_{I}^{(2)}(x')dx'\Big) \\[-2mm]
& \qquad\qquad\prod_{j=1}^{n-1}\Big(\int_{\mathbb{R}}F_{1,j}(x,y')\mathbf{b}_{J}^{(j)}(y')dy'\Big) \,\mathbf{a}_{I}^{(1)}(x)\mathbf{b}_{J}^{(n)}(y)
\end{align*}
from the range
$$ \textstyle\frac{1}{p_{2,n}}<\frac{1}{2}, \quad \frac{1}{p_{1,1}},\ldots,\frac{1}{p_{1,n-1}}<\frac{1}{2n-2},
\quad \frac{1}{p}=\frac{1}{p_{2,n}}+\sum_{j=1}^{n-1}\frac{1}{p_{1,j}}>1-\frac{1}{2n-2} $$
to a range that includes
\begin{equation}\label{eqnewregion2}
\textstyle\frac{1}{p_{2,n}},\frac{1}{p_{1,1}},\ldots,\frac{1}{p_{1,n-1}}<\frac{1}{n},
\quad \frac{1}{p}=\frac{1}{p_{2,n}}+\sum_{j=1}^{n-1}\frac{1}{p_{1,j}}>1-\frac{1}{2n-2} .
\end{equation}
We do that by performing the Calder\'{o}n-Zygmund decomposition in each fiber $F_{1,j}(x,\cdot)$ for $j=1,\ldots,n-1$ respectively.
In $j$-th step that extends the range of weak $\mathrm{L}^p$ estimates by raising $p_{1,j}^{-1}$ all the way to $1$.
The intersection of the new range with (\ref{eqnewregion2}) lies inside the Banach simplex,
so real multilinear interpolation actually provides strong estimates there.
In $n-1$ steps we recover the whole region (\ref{eqnewregion2}).

\subsection{A counterexample}
\label{subseccounterex}
Finally, let us show that the degenerate case $m=1$ gives objects that need not be bounded
just by assuming all other conditions of Theorem \ref{theoremmain}.
Since $G$ can have no isolated vertices, there must be an edge between $x_1$ and each of the vertices $y_j$,
so $E=\{(1,1),\ldots,(1,n)\}$.

If $m=n=1$, there can be no open range of estimates, so suppose $m=1<n$.
The form can be written as
$$ \Lambda_{E}(F_{1,1},\ldots,F_{1,n}) = \int_{\mathbb{R}} \Big(\int_{\mathbb{R}^n} K(x_1,y_1,\ldots,y_n)
\prod_{j=1}^{n} F_{1,j}(x_1,y_j) \,dy_1\ldots dy_n\Big) dx_1 . $$
One can notice that the inner integral has the symmetry of a dyadic one-dimensional multilinear Calder\'{o}n-Zygmund form for each fixed $x_1$.
Still, we claim that the two-dimensional conditions (\ref{eqt1wbpcondition}) and (\ref{eqt1condition})
are not enough to have any $\mathrm{L}^p$ bounds for $\Lambda_E$.
Instead one should impose one-dimensional testing conditions in each fiber of $K$.

To present the counterexample we take the kernel
$$ K(x_1,y_1,\ldots,y_n) := \sum_{I\times J\in\mathcal{C}} \lambda_{I\times J}\, |I|^{1-n} \mathbf{h}_{I}(x_1)
\mathbf{1}_{J}(y_1) \ldots \mathbf{1}_{J}(y_n) , $$
where the coefficients $\lambda=(\lambda_{Q})_{Q\in\mathcal{C}}$ are given by
$$ \lambda_{I\times J} := \left\{\begin{array}{cl} 1 & \textrm{ if } I=[0,2^{-k}) \textrm{ and } J\subseteq[0,1)
\textrm{ for some } k\in\{0,1,\ldots,r-1\} , \\
0 & \textrm{ otherwise}  \end{array}\right.$$
and $r$ is a positive integer.
Using only $|\lambda_{I\times J}|\leq 1$ we get
$$ \big|K(x_1,y_1,\ldots,y_n)\big| \leq \sum_{\substack{I\times J\in\mathcal{C}\\ x_1\in I,\ y_1,\ldots,y_n\in J}}\!\!|J|^{1-n}
\lesssim \Big(\min_{\substack{J\in\mathcal{D}\\ y_1,\ldots,y_n\in J}}\!\!|J|\,\Big)^{1-n}
\leq \Big(\max_{1\leq j_1<j_2\leq n}|y_{j_1}\!-y_{j_2}|\Big)^{1-n} , $$
under conventions $\min\emptyset=\infty$ and $\frac{1}{\infty}=0$,
which confirms the size estimate (\ref{eqczkernelest}) with $m=1$.
It is also immediate to verify
$$ |\Lambda_{E}(\mathbf{1}_{Q},\ldots,\mathbf{1}_{Q})| = \bigg|\sum_{Q'\in\mathcal{C} : Q'\supsetneqq Q} \lambda_{Q'}
|Q'|^{(1-n)/2} |Q|^{(n+1)/2}\bigg| \leq \frac{|Q|(\frac{1}{2})^{n-1}}{1-{(\frac{1}{2})}^{n-1}} \leq |Q| $$
for each $Q\in\mathcal{C}$ and
$$ \|T_{1,j}(\mathbf{1}_{\mathbb{R}^2},\ldots,\mathbf{1}_{\mathbb{R}^2})\|_{\mathrm{BMO}(\mathbb{R}^2)}
= \Big\|\sum_{I\times J\in \mathcal{C}} \!\lambda_{I\times J}\,\mathbf{h}_{I}\!\otimes\!\mathbf{1}_{J}\Big\|_{\mathrm{BMO}(\mathbb{R}^2)}
= \|\lambda\|_{\mathrm{bmo}} = \Big(\sum_{k=0}^{r-1} 2^{k} \big(2^{-k}\big)^2\Big)^{1/2} \!\leq \sqrt{2} $$
for $1\leq j\leq n$.
Thus the conditions of Theorem \ref{theoremmain} are fulfilled.

On the other hand, for each $j$ we define
$$ F_{1,j}(x,y) := \left\{\begin{array}{cl} \big(\frac{2^\ell}{\ell(\ell+1)}\big)^{1/n}
& \textrm{for } x\in[2^{-\ell},2^{-\ell+1}),\ y\in[0,1),\ \ell\in\{1,2,\ldots,r\} , \\
0 & \textrm{otherwise} , \end{array}\right. $$
and substitute into (\ref{eqentangledform}) to obtain
\begin{align*}
\Lambda_{E}(F_{1,1},\ldots,F_{1,n})
& = \sum_{k=0}^{r-1} 2^k (2^{-k})^{1-n} (2^{-k})^n \Big(\sum_{\ell=k+2}^{r}\frac{1}{\ell(\ell+1)} - \frac{1}{(k+1)(k+2)}\Big) \\
& = \sum_{k=0}^{r-1}\frac{1}{k+2} - \frac{r-1}{r+1} - 1 + \frac{1}{r+2} \longrightarrow \infty \quad\textrm{as } r\to\infty .
\end{align*}
Since
$$ \|F_{1,j}\|_{\mathrm{L}^n} = \Big(\sum_{\ell=1}^{r} 2^{-\ell} \frac{2^{\ell}}{\ell(\ell+1)}\Big)^{1/n}
= \Big(1-\frac{1}{r+1}\Big)^{1/n} \leq 1 , $$
we see that Bound (\ref{eqt1estimate}) does not hold when $p_{1,1}\!=\!\ldots\!=\!p_{1,n}\!=\!n$.
Indeed, the estimate cannot hold for any choice of the exponents, because symmetry and interpolation would then
recover the above ``central point'' bound.

\section{Necessity of the testing conditions}
\label{sectionnecessity}

We begin this section by reducing Theorem \ref{theoremreformulated} to Part (a) of Theorem \ref{theoremmain}.
\begin{proof}[Proof of Theorem \ref{theoremreformulated}]
Assume that Conditions (\ref{eqt1restricted}) are satisfied.
Let $r>0$ be large enough so that $K$ is supported in $[-r,r]^{m+n}$.

Fix $(u,v)\in E$ and take a square $Q=I\times J\in\mathcal{C}$.
Let us temporarily denote
$$ \mathcal{C}(Q) := \big\{ Q'\in\mathcal{C} \,:\, |Q'|=|Q|, \  Q' \textrm{ intersects } [-r,r]\times[-r,r] \big\} , $$
which is obviously a finite collection.
We start by decomposing
\begin{align*}
& T_{u,v}\big((\mathbf{1}_{\mathbb{R}^2})_{(i,j)\in E\setminus\{(u,v)\}}\big)(x_u,y_v)\mathbf{1}_{Q}(x_u,y_v) \\
& = \sum_{Q_{i,j}\in\mathcal{C}(Q)\mathrm{\,for\,each\,}(i,j)\in E\setminus\{(u,v)\}}
T_{u,v}\big((\mathbf{1}_{Q_{i,j}})_{(i,j)\in E\setminus\{(u,v)\}}\big)(x_u,y_v)\mathbf{1}_{Q}(x_u,y_v) \\
& = \sum_{Q_{i,j}\in\mathcal{C}(Q)\mathrm{\,for\,each\,}(i,j)\in E\setminus\{(u,v)\}}
\int_{\mathbb{R}^{m+n-2}} K(x_1,\ldots,x_m,y_1,\ldots,y_n) \prod_{(i,j)\in E} \!\mathbf{1}_{Q_{i,j}}(x_i,y_j) \,\prod_{i\neq u}dx_i \,\prod_{j\neq v}dy_j ,
\end{align*}
where for convenience we write $Q_{u,v}:=Q$.
Due to the term $\prod_{(i,j)\in E} \!\mathbf{1}_{Q_{i,j}}(x_i,y_j)$ the above summand
can be nonzero only when it is of the form
\begin{equation}\label{eqconstantterms}
\int_{\mathbb{R}^{m+n-2}} K(x_1,\ldots,x_m,y_1,\ldots,y_n)
\,\prod_{i=1}^{m} \mathbf{1}_{I^{(i)}}(x_i) \prod_{j=1}^{n} \mathbf{1}_{J^{(j)}}(y_j) \,\prod_{i\neq u}dx_i \,\prod_{j\neq v}dy_j
\end{equation}
for some dyadic intervals $I^{(1)},\ldots,I^{(m)},J^{(1)},\ldots,J^{(n)}$ of the same length and such that $I^{(u)}\times J^{(v)}=Q$.
Recall that $K$ is constant on all dyadic cubes in $\mathbb{R}^{m+n}$ that are disjoint from the diagonal $D$.
Thus, the expression in (\ref{eqconstantterms}) is constant for all $(x_u,y_v)\in Q$, except possibly when
$I^{(i)}=I$ for each $i$ and $J^{(j)}=J$ for each $j$, in which case (\ref{eqconstantterms}) becomes simply
\begin{align*}
& \int_{\mathbb{R}^{m+n-2}} K(x_1,\ldots,x_m,y_1,\ldots,y_n) \prod_{(i,j)\in E} \!\mathbf{1}_{Q}(x_i,y_j) \,\prod_{i\neq u}dx_i \,\prod_{j\neq v}dy_j \\
& = T_{u,v}\big((\mathbf{1}_{Q})_{(i,j)\in E\setminus\{(u,v)\}}\big)(x_u,y_v)\mathbf{1}_{Q}(x_u,y_v) .
\end{align*}
This discussion leads us to
\begin{align}
& T_{u,v}(\mathbf{1}_{\mathbb{R}^2},\ldots,\mathbf{1}_{\mathbb{R}^2})(x,y)
- \frac{1}{|Q|}\int_{Q} T_{u,v}(\mathbf{1}_{\mathbb{R}^2},\ldots,\mathbf{1}_{\mathbb{R}^2}) \nonumber \\
& =\, T_{u,v}(\mathbf{1}_{Q},\ldots,\mathbf{1}_{Q})(x,y)
- \frac{1}{|Q|}\int_{Q} T_{u,v}(\mathbf{1}_{Q},\ldots,\mathbf{1}_{Q}) \quad\textrm{for every } (x,y)\in Q . \label{eqequalitymodconstant}
\end{align}
Combining this equality with Conditions (\ref{eqt1restricted}) we obtain
$$ \frac{1}{|Q|} \int_{Q} \Big| T_{u,v}(\mathbf{1}_{\mathbb{R}^2},\ldots,\mathbf{1}_{\mathbb{R}^2})
- \frac{1}{|Q|}\int_{Q} T_{u,v}(\mathbf{1}_{\mathbb{R}^2},\ldots,\mathbf{1}_{\mathbb{R}^2}) \Big|
\leq \frac{2}{|Q|}\int_{Q} \big|T_{u,v}(\mathbf{1}_{Q},\ldots,\mathbf{1}_{Q})\big| \lesssim 1 . $$
It is well-known (as an easy consequence of the John-Nirenberg inequality) that the quantity
$$ \sup_{Q\in\mathcal{C}} \frac{1}{|Q|} \int_{Q} \Big| F - \frac{1}{|Q|}\int_{Q}F \Big| $$
is comparable with $\|F\|_{\mathrm{BMO}(\mathbb{R}^2)}$, see \cite[Section 3]{AHMTT}.
That way we have established Condition (\ref{eqt1condition}).

Verification of Condition (\ref{eqt1wbpcondition}) using (\ref{eqt1restricted}) is easy,
as for some $(u,v)\in E$ and an arbitrary $Q\in\mathcal{C}$,
$$ |\Lambda_{E}(\mathbf{1}_{Q},\ldots,\mathbf{1}_{Q})|
= \Big|\int_{\mathbb{R}^2} T_{u,v}(\mathbf{1}_{Q},\ldots,\mathbf{1}_{Q})\,\mathbf{1}_{Q}\Big| \\
\leq \|T_{u,v}(\mathbf{1}_{Q},\ldots,\mathbf{1}_{Q})\|_{\mathrm{L}^{1}(Q)} \lesssim |Q| . $$
Therefore, Theorem \ref{theoremmain} (a) can be applied and this completes the proof.
\end{proof}

In order to prove Part (b) of Theorem \ref{theoremmain}, it is now enough to reduce its hypotheses to Conditions (\ref{eqt1restricted}),
which imply (\ref{eqt1wbpcondition}) and (\ref{eqt1condition}), as we have just shown.
\begin{proof}[Proof of Theorem \ref{theoremmain} (b)]
Suppose that Estimate (\ref{eqt1estimate}) holds with some choice of exponents $p_{i,j}$.
Fix $(u,v)\in E$, take $Q\in\mathcal{C}$, choose $F_{i,j}=\mathbf{1}_{Q}$ for $(i,j)\in E\setminus\{(u,v)\}$, and leave $F_{u,v}$ arbitrary.
Let $p'_{u,v}$ denote the conjugated exponent of $p_{u,v}$.
By (\ref{eqduality}) and (\ref{eqt1estimate}),
$$ \Big|\int_{\mathbb{R}^2} T_{u,v}(\mathbf{1}_{Q},\ldots,\mathbf{1}_{Q})\, F_{u,v}\Big|
\,\lesssim\, \|F_{u,v}\|_{\mathrm{L}^{p_{u,v}}} \!\!\prod_{(i,j)\neq (u,v)}\!\! \|\mathbf{1}_{Q}\|_{\mathrm{L}^{p_{i,j}}}
= \|F_{u,v}\|_{\mathrm{L}^{p_{u,v}}} |Q|^{1/p'_{u,v}} , $$
so duality implies
$$ \| T_{u,v}(\mathbf{1}_{Q},\ldots,\mathbf{1}_{Q}) \|_{\mathrm{L}^{p'_{u,v}}\!(Q)} \lesssim |Q|^{1/p'_{u,v}} . $$
Furthermore, Jensen's inequality gives
$$ \frac{1}{|Q|}\int_{Q} \big|T_{u,v}(\mathbf{1}_{Q},\ldots,\mathbf{1}_{Q})\big|
\leq \Big(\frac{1}{|Q|}\int_{Q} \big|T_{u,v}(\mathbf{1}_{Q},\ldots,\mathbf{1}_{Q})\big|^{p'_{u,v}}\Big)^{1/p'_{u,v}} \lesssim 1 , $$
which is exactly (\ref{eqt1restricted}).
\end{proof}

Let us conclude with a comment that the short proofs given in this section are adaptations of classical arguments from \cite{Ste}.
If the perfect cancellation of $K$ was replaced with the standard H\"{o}lder continuity condition,
the difference of the two sides in (\ref{eqequalitymodconstant}) would only be a bounded function, which
would still be enough to follow the reasoning of the corresponding part
of the proof.

\end{document}